\documentclass[draft,11pt]{amsart}
\usepackage{inputenc,amsmath,amssymb,amsthm,xcolor,url,mathtools,microtype}
\title{Pursuing polynomial bounds on torsion}
\usepackage[text={144mm,220mm},centering]{geometry} 
\date{\today}

\numberwithin{equation}{section}

\begin{document}
\author{Pete L. Clark}
\author{Paul Pollack}

\def\tors{{\rm tors}}
\def\Z{\mathbb{Z}}
\def\Q{\mathbb{Q}}
\newcommand{\R}{\mathbb{R}}
\newcommand{\C}{\mathbb{C}}
\def\Ss{\mathcal{S}}
\def\l{\mathfrak{l}}
\newcommand{\PP}{\mathbb{P}}
\newcommand{\sep}{\operatorname{sep}}
\def\UNIFORM{\texttt{UNIFORM}}
\def\MINDEGREE{\texttt{MINDEGREE}}
\newcommand{\ra}{\rightarrow}
\newcommand{\ord}{\operatorname{ord}}
\newcommand{\N}{\mathbb{N}}
\renewcommand{\gg}{\mathfrak{g}}
\newcommand{\GL}{\operatorname{GL}}
\newcommand{\SL}{\operatorname{SL}}
\newcommand{\PSL}{\operatorname{PSL}}
\newcommand{\Zhat}{\widehat{\Z}}
\newcommand{\Aut}{\operatorname{Aut}}
\newcommand{\End}{\operatorname{End}}
\newcommand{\cyc}{\operatorname{cyc}}
\newcommand{\lcyc}{\ell\text{-}\!\operatorname{cyc}}
\newcommand{\SI}{\operatorname{SI}}
\renewcommand{\ra}{\rightarrow}
\newcommand{\Ker}{\operatorname{Ker}}

\newtheorem{lemma}{Lemma}[section]
\newtheorem{thm}[lemma]{Theorem}

\newtheorem{conj}[lemma]{Conjecture}
\newtheorem{cor}[lemma]{Corollary}

\theoremstyle{remark}
\newtheorem{remark}[lemma]{Remark}

\makeatletter
\def\subsection{\@startsection{subsection}{2}%
  \z@{.5\linespacing\@plus.7\linespacing}{.3\linespacing}%
  {\normalfont\bfseries}}
\makeatother


\begin{abstract}
We show that for all $\epsilon > 0$, there is a constant $C(\epsilon) > 0$ such that
for all elliptic curves $E$ defined over a number field $F$ with $j(E) \in \Q$ we have \[ \# E(F)[\tors] \leq C(\epsilon) [F:\Q]^{5/2 + \epsilon}.\]
We pursue further bounds on the size of the torsion subgroup of an elliptic curve over a number
field $E_{/F}$ that are polynomial in $[F:\Q]$ under restrictions on $j(E)$.  We give an unconditional result for
$j(E)$ lying in a fixed quadratic field that is not imaginary of class number one as well as two further results, one conditional on GRH and one conditional
on the strong boundedness of isogenies of prime degree for non-CM elliptic curves.
\end{abstract}

\maketitle


\section*{Notation}
\noindent
Let $\mathcal{P}$ be the set of prime numbers.  For a commutative group $G$, we denote the subgroup of elements of order dividing $n$ by $G[n]$ and the torsion subgroup -- i.e., the subgroup of elements of finite order -- by $G[\tors]$.  For $\mathfrak{s} \subset \mathcal{P}$,
we let $G[\mathfrak{s}^{\infty}]$ denote the subgroup of $G[\tors]$ of elements with order divisible only by primes in $\mathfrak{s}$.
If for a commutative group $G$ we have $G = G[n]$ for some $n \in \Z^+$ -- as is the case if $G$ is finite -- then the least such $n$ is the \textbf{exponent}
$\exp G$ of $G$.  \
\\ \indent
For a field $F$, let $\overline{F}$ be an algebraic closure.  We denote by $\gg_{F} = \Aut(\overline{F}/F)$
the absolute Galois group of $F$.  For $n \in \Z^+$ and a field $F$ of characteristic $0$, let $F(\mu_n)$ be the field obtained by adjoining to $F$ the
$n$th roots of unity, and let $F^{\cyc} = \bigcup_n F(\mu_n)$.  Let $\chi_{\ell}\colon \gg_F \ra \Z_{\ell}^{\times}$
be the $\ell$-adic cyclotomic character, and put
\[ F^{\lcyc} \coloneqq \overline{F}^{\Ker \chi_{\ell}} =  \bigcup_{n \in \Z^+} F(\mu_{\ell^n}). \]

\section{Introduction}

\subsection{Known bounds on the torsion subgroup}
\noindent
For an elliptic curve $E$ defined over a number field $F$, the torsion subgroup $E(F)[\tors]$ is finite.  Moreover, by Merel's strong\footnote{Here and throughout this paper we use the term ``strong'' to mean a bound that is uniform across elliptic curves defined over number fields of any fixed degree.}
 uniform boundedness theorem \cite{Merel96}, as we range over all degree $d$ number fields
$F$ and all elliptic curves $E_{/F}$, we have
\[ T(d) = \sup \# E(F)[\tors] < \infty. \]
Merel's work gives an explicit upper bound on $T(d)$, which was improved by Oesterl\'e (unpublished, but see \cite{Derickx16}) and Parent \cite{Parent99}.  These lie more than an exponential away from the
best known lower bound, due to Breuer \cite{Breuer10}:
\[ \inf_d \frac{T(d)}{d \log \log d} > 0. \]
Various authors have conjectured that Breuer's lower bound is esentially sharp.

\begin{conj}
\label{STRONGCONJ}
We have $T(d) = O(d \log \log d)$.
\end{conj}
\noindent
The following weaker conjecture is more widely believed.

\begin{conj}
\label{WEAKERCONJ}
$T(d)$ is \emph{polynomially bounded}: there is $B > 1$ such that
\begin{equation*}
 T(d) = O(d^B).
\end{equation*}
\end{conj}
\noindent
Even Conjecture \ref{WEAKERCONJ} seems to lie out of current reach.  Since the work of Merel, Osterl\'e and Parent, progress
on understanding the asymptotic behavior of $T(d)$ has come (only) by restricting the class of elliptic curves under
consideration.  For instance if we restrict to the case in which $E$ has \emph{complex multiplication} (CM) --  and write
$T_{\operatorname{CM}}(d)$ in place of $T(d)$ when this restriction is made --  breakthrough
work of Silverberg \cite{Silverberg88, Silverberg92} gave the asymptotically correct upper bound on the exponent,
and recent work by the present authors \cite{CP15,CP17} shows
\begin{equation*} \limsup_{d} \frac{T_{\operatorname{CM}}(d)}{d \log \log d} = \frac{e^{\gamma} \pi}{\sqrt{3}}. \end{equation*}
If we instead restrict to the class of elliptic curves with \emph{integral moduli} -- i.e., with $j$-invariant lying
in the ring $\overline{\Z}$ of algebraic integers -- and write $T_{\operatorname{IM}}(d)$ in place of $T(d)$, then
Hindry-Silverman showed \cite{HS99}
\begin{equation*} T_{\operatorname{IM}}(d) = O(d \log d). \end{equation*}

\subsection{In pursuit of polynomial bounds}
\noindent
In this paper we will pursue polynomial bounds on the size of the torsion subgroup in a different kind of restricted regime.
We begin by stating the following result, which conveys the flavor with a minimum of technical hypotheses.

\begin{thm}
\label{THM0}
Let $\epsilon > 0$.  Then there is $C=C(\epsilon)$ such that: for all degree $d$ number fields $F$ and
all elliptic curves $E_{/F}$ arising via base extension from an elliptic curve $(E_0)_{/\Q}$, we have
\begin{equation}
\label{THM0EQ}
\exp E(F)[\tors] \leq C
 d^{3/2 + \epsilon}.
\end{equation}
\end{thm}
\noindent
Now we make some comments:
\begin{itemize}
\item Of course we can assume that $E$ does not have CM.
\item For any elliptic curve $E$ over a number field $F$, we have
\begin{equation}
\label{CRUDEEQ}
\# E(F)[\tors] \mid (\exp E(F)[\tors])^2.
\end{equation}
Thus, Theorem \ref{THM0}
implies a bound of $O_{\epsilon}(d^{3+\epsilon})$ on the size of the torsion subgroup itself.  Later we will give an improvement
of this bound.
\item An easy quadratic twisting argument allows us to establish the bound
(\ref{THM0EQ}) as we range over all elliptic curves $E_{/F}$ with $j(E) \in \Q$.
This serves to motivate the type of further result we would like to prove.
\end{itemize}
Let $d_0 \in \Z^+$.  For a positive integer $d$ that is divisible by $d_0$, let $T_{d_0}(d)$ be the supremum of
$\# E(F)[\tors]$ as $E$ ranges over all elliptic curves defined over a degree $d$ number field $F$ such that
$[\Q(j(E)):\Q] = d_0$.

\begin{conj}
\label{CONJ0}
For each $d_0 \in \Z^+$, there are $B = B(d_0)$ and $C=C(d_0)$ such that
\[ T_{d_0}(d) \le C d^B. \]
\end{conj}
\noindent
Theorem \ref{THM0} gives the case $d_0 = 1$ of Conjecture \ref{CONJ0}.  At
present we cannot prove Conjecture \ref{CONJ0} unconditionally for any $d_0 \geq 2$, but we can make some
progress in this direction, as shown by the following results.

\begin{thm}
\label{THM1}
Let $F_0$ be a quadratic number field that is not imaginary quadratic of class number one.  (Thus,
the discriminant of $F_0$ is not one of $-3,-4,-7,-8,-11,$\\ $-19,-43,-67,-163$.)  Then for all $\epsilon > 0$, there is $C= C(\epsilon,F_0)$ such that if $F$ is a degree $d$ number field and $E_{/F}$
is an elliptic curve with $j(E) \in F_0$, we have
\[ \exp E(F)[\tors] \leq C d^{3/2+\epsilon}\quad\text{and}\quad\# E(F)[\tors] \leq C d^{5/2+\epsilon}. \]
\end{thm}

\begin{thm}
\label{THM2}
Assume the Generalized Riemann Hypothesis (GRH).\footnote{Throughout, we use GRH to mean the Riemann Hypothesis for all Dedekind zeta functions.}  Let $F_0$ be a number field that does not contain the Hilbert
class field of any imaginary quadratic field.  Then for all $\epsilon > 0$, there is $C = C(\epsilon,F_0)$ such that if $F$ is a degree $d$ number field and $E_{/F}$
is an elliptic curve with $j(E) \in F_0$, we have
\[ \exp E(F)[\tors] \leq C d^{3/2+\epsilon}\quad\text{and}\quad\# E(F)[\tors] \leq C d^{5/2+\epsilon}. \]
\end{thm}
\noindent
For $d_0 \in \Z^+$, we introduce a hypothesis $\operatorname{SI}(d_0)$ defined as follows.
\[ \operatorname{SI}(d_0)\!:\quad\text{\begin{minipage}[t]{30em}
There is prime $\ell_0 = \ell_0(d_0)$ such that for all primes $\ell > \ell_0$, the modular curve $X_0(\ell)$ has 
no noncuspidal non-CM points of degree $d_0$.  
\end{minipage}} \]


\begin{thm}
\label{THM3}
If $\operatorname{SI}(d_0)$ holds, then for all $\epsilon > 0$, there is $C = C(\epsilon,d_0)$ with
\[ T_{d_0}(d) \le C d^{\frac{5}{2}+\epsilon}. \]
\end{thm}

\begin{remark}
This paper is cognate to another work \cite{CMP17}, written in parallel, giving ``typical bounds'' on $\# E(F)[\tors]$
for an elliptic curve $E_{/F}$, under the same hypotheses as Theorems \ref{THM1}, \ref{THM2} and \ref{THM3}.
\end{remark}

\subsection{Strategy of the proofs}
\noindent
In \cite{Arai08}, Arai showed that for each fixed prime $\ell$ and number field $F$, as we range over all non-CM elliptic
curves $E_{/F}$ there is a uniform upper bound on the index of the image of
the $\ell$-adic Galois representation.  In $\S$2 we prove the \emph{strong} form of this theorem by showing that
the conclusion still holds as we range over all non-CM elliptic curves defined over all number fields
of any fixed degree (Theorem \ref{STRONGARAITHM}).
\\ \indent  A uniform upper bound on the index of the \emph{adelic} Galois representation as we range over all non-CM elliptic
curves defined over number fields of fixed degree $d_0$ would be -- to say the least! -- desirable.  It implies $\operatorname{SI}(d_0)$ but is so much stronger that the $d_0 = 1$ case was raised as an open problem
in \cite{Serre72}, has guided most subsequent work in the field, and remains open.  Our approach to the $\ell$-adic
version exploits the finiteness properties that $\GL_2(\Z_{\ell})$ enjoys by virtue of being an $\ell$-adic
analytic group -- finiteness properties that $\GL_2(\widehat{\Z})$ certainly does not possess.  \\ \indent
From Theorem \ref{STRONGARAITHM} we deduce Theorem \ref{ARAICOR}: for each finite set of primes $\mathfrak{s}$, the quantity $\exp E(F)[\mathfrak{s}^{\infty}]$ is bounded by a polynomial in $[F:\Q(j(E))]$.  Thus in order to bound $\exp E(F)[\tors]$ it suffices to bound $\exp E(F)[\ell^{\infty}]$
for all \emph{sufficiently large} primes.
\\ \\
The next ingredient is the following striking recent result.

\begin{thm}[Lozano-Robledo]
\label{BIGALVAROTHM}
Let $\ell > 2$ be a prime.  Let $F_0$ be a number field, $\mathfrak{l}$ a prime ideal of $\Z_{F_0}$ lying over $\ell$, and $e(\mathfrak{l}/\ell)$ the ramification index.  Let $E_{/F_0}$ be a non-CM
elliptic curve, and let $a \in \Z^+$ be such that $E$ admits no $F_0$-rational cyclic isogeny of degree $\ell^a$.
Let $n \geq a$, and let $P \in E(\overline{F_0})$ have order $\ell^n$.  Then there is an integer
$1 \leq c \leq 12 e(\mathfrak{l}/\ell)$ and a prime $\mathcal{L}$ of $F_0(P)$ lying over $\mathfrak{l}$ such that
the ramification index $e(\mathcal{L}/\mathfrak{l})$ is divisible by either $\frac{\varphi(\ell^n)}{\gcd(\varphi(\ell^n),c \ell^{a-1})}$ or by $\ell^{n-a+1}$.
\end{thm}
\begin{proof} This is a simplified form of \cite[Thm. 2.1]{UNIFORM}.
\end{proof}
\noindent
To apply Theorem \ref{BIGALVAROTHM} to get uniform bounds on $\exp E(F)[\tors]$, we need finiteness results for rational $\ell$-isogenies.  Thus Hypothesis $\operatorname{SI}(d_0)$ intervenes naturally.

\begin{cor}
\label{KEYCORA}
Let $d_0 \in \Z^+$, and assume hypothesis $\operatorname{SI}(d_0)$.  There is a prime $\ell_0 = \ell_0(d_0)$ such that:
for all number fields $F_0/\Q$ of degree $d_0$ and all primes $\ell > \ell_0$, if $\mathfrak{l}$ is a prime ideal of $\Z_{F_0}$
lying over $\ell$ and $E_{/F_0}$ is a non-CM elliptic curve, then for all $n \in \Z^+$, if
$P \in E(\overline{F_0})$ is a point of order $\ell^n$, then there is $1 \leq c \leq 12 d_0$ and a prime
$\mathcal{L}$ of $F_0(P)$ lying over $\mathfrak{l}$ such that $e(\mathcal{L}/\mathfrak{l})$ is divisible by either $\frac{\varphi(\ell^n)}{\gcd(\varphi(\ell^n),c)}$ or by $\ell^n$.
\end{cor}
\begin{proof}
This follows from Theorem \ref{BIGALVAROTHM} and the bound $e(\mathfrak{l}/l) \leq d_0$.
\end{proof}
\noindent
To proceed
without assuming $\operatorname{SI}(d_0)$ we need to restrict to weaker statements that are known or conditionally known.
We make use of the following prior results.

\begin{thm}[Mazur \cite{Mazur78}]
\label{MAZURTHM}
The hypothesis $\operatorname{SI}(1)$ holds with $\ell_0(1) = 37$.
\end{thm}
\noindent
Theorems \ref{THM3} and \ref{MAZURTHM} imply Theorem \ref{THM0}.

\begin{thm}[Momose \cite{Momose95}]
\label{MOMOSETHM}
Let $F_0$ be a quadratic field that is not imaginary quadratic of class number $1$.  There is a prime
number $\ell_0 = \ell_0(F_0)$ such that for all primes $\ell > \ell_0$, no elliptic curve $E_{/F_0}$ admits
an $F_0$-rational isogeny of degree $\ell$.
\end{thm}

\begin{thm}[Larson-Vaintrob {\cite[Cor. 6.5]{LV14}}]
\label{LARSONVAINTROBTHM}
Let $F_0$ be a number field that does not contain the Hilbert class field of any imaginary quadratic field.  If
the Generalized Riemann Hypothesis (GRH) holds, then the set of prime numbers $\ell$ such that some
elliptic curve $E_{/F_0}$ admits an $F_0$-rational $\ell$-isogeny is finite.
\end{thm}

\begin{cor}
\label{KEYCOR}
\label{KEYCORB}
Let $F_0$ be a number field that does not contain the Hilbert class field of any imaginary quadratic field.
If $[F_0:\Q] \geq 3$, assume GRH.  Then there is a prime $\ell_0 = \ell_0(F_0)$ such that: for all primes $\ell > \ell_0$,
if $\mathfrak{l}$ is a prime ideal of $\Z_{F_0}$ lying over $\ell$, and $E_{/F_0}$ is an elliptic curve, then for all
$n \in \Z^+$, if $P \in E(\overline{F_0})$ is a point of order $\ell^n$, then there is an integer $1 \leq c \leq
12e(\mathfrak{l}/l)$ and a prime $\mathcal{L}$ of $F_0(P)$ lying over $\mathfrak{l}$ such that $e(\mathcal{L}/\mathfrak{l})$
is divisible by either $\frac{\varphi(\ell^n)}{\gcd(\varphi(\ell^n),c)}$ or by $\ell^n$.
\end{cor}
\begin{proof}
Combine Theorems \ref{BIGALVAROTHM}, \ref{MOMOSETHM} and \ref{LARSONVAINTROBTHM}.
\end{proof}
\noindent
In $\S$3.2 we use Corollaries \ref{KEYCORA} and \ref{KEYCORB} to get the polynomial bounds on $\exp E(F)[\tors]$ of Theorems \ref{THM1}, \ref{THM2} and \ref{THM3}.  Via (\ref{CRUDEEQ}) this immediately gives a polynomial bound on $\# E(F)[\tors]$.   However in $\S$3.3 we improve this bound using an analysis of
cyclotomic characters, completing the proofs of Theorems \ref{THM0}, \ref{THM1}, \ref{THM2} and \ref{THM3}.


\section{Bounded index results for the $\ell$-adic Galois representation}
\subsection{Statement of the strong Arai theorem}
\noindent
Let $F$ be a number field, and let $E_{/F}$ be a non-CM elliptic curve.
Let
\[ \widehat{\rho}\colon \gg_{F} \ra \Aut T E \cong \GL_2(\Zhat) \]
denote the adelic Galois representation of $E_{/F}$, and for a prime $\ell$, let
\[ \rho_{\ell^{\infty}}\colon \gg_{F} \ra \Aut T_{\ell} E \cong \GL_2(\Z_{\ell}) \]
denote the $\ell$-adic Galois representation of $E_{/F}$.
Then $\det \rho_{\ell^{\infty}} = \chi_{\ell}$ and $F^{\lcyc} = \overline{F}^{\Ker \chi_{\ell}}$, so
\begin{equation}
\label{SPECIALEQ}
 \rho_{\ell^{\infty}}(\gg_{F^{\lcyc}}) = \rho_{\ell^{\infty}}(\gg_F) \cap \SL_2(\Z_{\ell}).
\end{equation}
By a result of Serre \cite{Serre72}, the image
$\widehat{\rho}(\gg_{F})$ is open in $\GL_2(\Zhat)$ -- equivalently, has finite index.  Thus for each
prime $\ell$ the $\ell$-adic image $\rho_{\ell^{\infty}}(\gg_{F})$ has finite index in $\GL_2(\Z_{\ell})$
and $\rho_{\ell^{\infty}}$ is surjective for all but finitely many primes $\ell$.  As mentioned above,
a uniform adelic open image theorem is the \emph{ultima Thule} of this field,
but Arai has proved a uniform $\ell$-adic open image theorem.

\begin{thm}[Arai \cite{Arai08}]
\label{ARAITHM}
Let $F$ be a number field, and let $\ell$ be a prime.  There is $I \in \Z^+$ such that for
every non-CM elliptic curve $E_{/F}$, the image $\rho_{\ell^{\infty}}(\gg_F)$ of the $\ell$-adic Galois representation
has index at most $I$ in $\GL_2(\Z_{\ell})$.
\end{thm}

\begin{remark}
Arai states Theorem \ref{ARAITHM} slightly differently.  For $n \in \Z^+$, put
\[ \mathcal{U}^{(n)} \coloneqq \operatorname{Ker} (\GL_2(\Z_{\ell}) \ra \GL_2(\Z/\ell^n \Z)). \]
Then each $\mathcal{U}^{(n)}$ is an open subgroup of $\GL_2(\Z_{\ell})$, and each open subgroup $\Gamma$ of
$\GL_2(\Z_{\ell})$ contains $\mathcal{U}^{(n)}$ for all sufficiently large $n$.  The least such $n$ is called the
\textbf{level} of $\Gamma$.  Then Arai proves: for a number field $F$ and a prime $\ell$, there is $n = n(F,\ell)$
such that for every non-CM elliptic curve $E_{/F}$, the level of $\rho_{\ell^{\infty}}(\gg_F)$ is at most $n$.  \\ \indent
The statement in terms of the level immediately implies the statement in terms of the index.  The reverse
implication holds because (cf. Lemma \ref{CLEVERGROUP}a)) the intersection of all open subgroups of $\GL_2(\Z_{\ell})$
of index at most $I$ is an open subgroup of $\GL_2(\Z_{\ell})$.
\end{remark}
\noindent
The main goal of this section is to prove the following \emph{strong} form of Arai's theorem.

\begin{thm}
\label{STRONGARAITHM}
Fix a prime number $\ell$ and a positive integer $d$.
\begin{itemize}
\item[a)] As we range over all non-CM elliptic curves $E_{/F}$ defined over number fields of degree $d$, there is an absolute
bound on the index of the image of the $\ell$-adic Galois representation $\rho_{\ell^{\infty}}(\gg_F)$ in $\GL_2(\Z_{\ell})$.
\item[b)] Moreover, for all but finitely many $j$-invariants we have
\[ [\GL_2(\Z_{\ell}):\rho_{\ell^{\infty}}(\gg_F)] \leq \frac{3200 d^2}{7}. \]
\end{itemize}
\end{thm}

\begin{remark}\mbox{ }
\begin{itemize}
\item[a)] Theorem \ref{STRONGARAITHM}a) is a quick consequence of results of Cadoret and Tamagawa \cite{CT12,CT13}. Namely, we apply \cite[Thm. 1.1]{CT13} with $k = \Q$ to the family of elliptic curves $E \ra X \coloneqq \PP^1 \setminus \{0,1728,\infty\}$ given by  \[ E_j: y^2 + xy - x^3 + \frac{36}{j-1728}x + \frac{1}{j-1728} = 0 \]
of \cite[\S5.1.3]{CT12} -- this family is geometrically Lie perfect by \cite[Thm. 5.1]{CT12}.  The conclusion is that, for each $d \in \Z^+$, for each fixed prime $\ell$ and positive integer $d$, there is $B(\ell,d) \in \Z^+$ such that for all but finitely many closed points
$j \in X$ of degree at most $d$, the index of the image of $\ell$-adic Galois representation on $(E_j)_{/\Q(j)}$ in $\GL_2(\Z_{\ell})$
is at most $B(\ell,d)$.  By Serre's open image theorem the result extends to all non-CM $j$-invariants of degree at most $d$ with
some absolute bound, say $\widetilde{B(\ell,d)}$.   For
any non-CM elliptic curve $E$ defined over a degree $d$ number field $F$, there is an elliptic curve $E_j$ in the above family and a
number field $K \supset F(j(E))$ with $[K:\Q] \leq 2d$ such that $(E_j)_{/K} \cong E_{/K}$.  The index of the image of the $\ell$-adic
Galois representation of $E_{/F}$ in $\GL_2(\Z_{\ell})$ is no larger than the index of the $\ell$-adic Galois representation of
$E_{/K}$ in $\GL_2(\Z_{\ell})$ and thus no larger than $2d \cdot \widetilde{B(\ell,d)}$. \item[b)] Rouse outlined a proof of Theorem \ref{STRONGARAITHM}a) on MathOverflow \cite{Rouse}.  His methods would yield a version
of Theorem \ref{STRONGARAITHM}b).
\item[c)] Theorem \ref{STRONGARAITHM}a) is sufficient for our applications.  Nevertheless we want to include a proof of
Theorem \ref{STRONGARAITHM}b).  First, it seems interesting that in a natural case\footnote{In \cite{CT12}, the authors
identify Arai's work as a motivation for their own.} of \cite[Thm. 1.1]{CT13} we can get a bound that is -- after omitting a finite
set of $j$-invariants that depends on $\ell$ and $d$ -- explicit
and independent of $\ell$.  Second, the proof of \cite[Thm. 1.1]{CT13} takes about $25$ pages, whereas the outline
of \cite{Rouse} is 13 lines.  Our argument is about 2.5 pages; readers may appreciate having a
proof of this intermediate length.  Finally, in \cite[Thm. 1.2]{CT13}, Cadoret-Tamagawa state a result of Frey \cite{Frey94} but
omit Frey's assumption that $X(k) \neq \varnothing$.  This is easily remedied by using a variant on Frey's result
from \cite{Clark09}, and our argument shows how to do this.
\end{itemize}
\end{remark}

\subsection{Group theoretic preliminaries}

\begin{lemma}
\label{CLEVERGROUP}
Let $\ell$ be a prime, and let $G$ be an infinite $\ell$-adic analytic group.
\begin{enumerate}
	\item[a)] $G$ is topologically finitely generated.  A subgroup of $G$ is open iff it has finite index.  For all $I \in \Z^+$, there are only finitely many index $I$ subgroups of $G$.
	\item[b)] Every open subgroup of $G$ has at least one and finitely many maximal proper open subgroups.
	\item[c)] Let $\mathcal{F}$ be a set of open subgroups of $G$ such that $\mathcal{F}$ contains all but finitely many open subgroups of $G$.  Then every element of $\mathcal{F}$ is contained in a
maximal element, and $\mathcal{F}$ has finitely many maximal elements.
\item[d)] For $I \in \Z^+$, the family $\mathcal{F}_I$ of open subgroups $H \subset G$ with $[G:H] > I$ satisfies
the hypotheses of part c) and thus has at least one and finitely many maximal elements.

\end{enumerate}
\end{lemma}
\begin{proof}
a) Lazard has shown that
 an $\ell$-adic analytic group
is topologically finitely generated and that a subgroup of $G$ is open iff it has finite index \cite{Lazard65}.  Moreover,
every topologically finitely generated profinite
group has only finitely many open subgroups of any given finite index \cite[Lemma 16.10.2]{FJ}.\footnote{Each finite index subgroup of a topologically finitely generated profinite group is open \cite{NS07}.} \medskip

\noindent b) Every nontrivial profinite group has a proper open subgroup, and any such group is contained in a maximal proper subgroup.  And every open subgroup $U$ of $G$ is again an $\ell$-adic analytic group, so by \cite[pp. 148--149]{S-LMW} the Frattini subgroup $\Phi(U)$ is open and thus $U$ has only finitely many maximal proper open subgroups.\medskip

\noindent c) We may assume $G \notin \mathcal{F}$.  Since $G$ is infinite and profinite, $\mathcal{F} \neq \varnothing$.  As for any group, the set of finite index subgroups of $G$, partially ordered under inclusion, satisfies the ascending chain
condition, hence so does $\mathcal{F}$ and every element of $\mathcal{F}$ is contained in a maximal element.  To show that there are only finitely many maximal elements of $\mathcal{F}$
it suffices to find a finite subset $\mathcal{S} \subset \mathcal{F}$ such that for every $K \in \mathcal{F}$
there is $H \in \mathcal{S}$ such that $K \subseteq H$, for then the maximal elements of $\mathcal{F}$ are the maximal
elements of $\mathcal{S}$.  \\ \indent
Let $\mathcal{S}$ be the set of elements $H \in \mathcal{F}$
such that $H$ is a maximal proper open subgroup of an open subgroup $P$ of $G$ such that $P \notin \mathcal{F}$.  By assumption the set of such subgroups $P$ is finite, so $\mathcal{S}$ is finite by part b).  Because $G \notin \mathcal{G}$, for $K \in \mathcal{F}$, the set
of subgroups $Q$ with $K \subsetneq Q \subseteq G$ and $Q \notin \mathcal{F}$ is finite and nonempty; choose
a minimal element $\underline{Q}$.  The set of subgroups $H$ with $K \subseteq H \subsetneq \underline{Q}$
is finite and nonempty; choose a maximal element $\overline{H}$.  Then $K \subset \overline{H}$ and $\overline{H} \in \mathcal{S}$.
\medskip
\noindent d) This is immediate from part a).
\end{proof}

\begin{remark}
It follows from \cite[pp. 148]{S-LMW} that the necessary and sufficient condition on a profinite group $G$ for all of the conclusions
of Lemma \ref{CLEVERGROUP} to hold for $G$ is that the Frattini subgroup $\Phi(G)$ of $G$ be open.
\end{remark}


\begin{lemma}
\label{NEWSLEMMA}
Let $F$ be a degree $d$ number field, and let $E_{/F}$ be an elliptic curve.  Let $\ell$ be a prime number.  Let
$G = \rho_{\ell^{\infty}}(\gg_F)$ be the image of the $\ell$-adic Galois representation on $E$ and let
$H = G \cap \SL_2(\Z_{\ell})$.  Then we have
\begin{equation}
\label{NEWSEQ}
[\GL_2(\Z_{\ell}):G] \leq d[\SL_2(\Z_{\ell}):H].
\end{equation}
\end{lemma}
\begin{proof} Step 1: Let $G$ be a group, let $N$ be a normal subgroup of $G$, and let $q\colon G \ra G/N$ be the quotient map.  Then we have
\begin{equation}
\label{LITTLEGROUP}
 [G:H] \leq [N:H \cap N][G/N:q(H)].
\end{equation}
Indeed, let $X \subset G$ have cardinality larger than $[N:H \cap N][G/N:q(H)]$.  By the Pigeonhole
Principle, there is $Y \subset X$ of cardinality larger than ${[N:H \cap N]} = {[HN:H]}$ such that
for all $y_1,y_2 \in Y$, we have $q(y_2 y_1^{-1}) \in q(H)$, so $y_2 y_1^{-1} \in q^{-1}(q(H)) = HN$.
So there are $y_1 \neq y_2$ such that $y_1 H = y_2 H$.
\medskip

\noindent Step 2: The $\ell$-adic cyclotomic character  $\chi_{\ell}\colon \gg_{\Q} \rightarrow \widehat{\Z}_{\ell}^{\times}$
is surjective, so for any degree $d$ number field $F$ and elliptic curve $E_{/F}$, we have $[\Z_\ell^{\times}:\det \rho_{\ell^{\infty}}(\gg_F)] \mid d$.
 Applying (\ref{LITTLEGROUP})
with $G = \GL_2(\Z_{\ell})$, $N = \SL_2(\Z_{\ell})$, $H = G$ and using (\ref{SPECIALEQ}), we get (\ref{NEWSEQ}).
\end{proof}

\noindent

\subsection{Proof of Theorem \ref{STRONGARAITHM}}
\noindent
Step 1: Put $I \coloneqq \left\lfloor 1600d^2/7 \right\rfloor$.  Let $F$ be a number field, let $E_{/F}$ be a non-CM elliptic curve, and let $G = \rho_{\ell^{\infty}}(\gg_F)$ be the image of the $\ell$-adic Galois representation on $E_{/F}$.  Some quadratic twist $E'$ of $E$ has $-1 \in G$.  Put $G' \coloneqq
\rho_{\ell^{\infty}}(\gg_F)$.  Then $[\GL_2(\Z_{\ell}):G] \leq 2[\GL_2(\Z_{\ell}):G']$, so it suffices to work with $E'$.  We may also assume that $[\GL_2(\Z_{\ell}):G'] > I$.  Applying Lemma \ref{CLEVERGROUP} with $G = \GL_2(\Z_{\ell})$ we get that $G'$ is contained in one of finitely many open subgroups $\Gamma \subset \GL_2(\Z_{\ell})$ with $[\GL_2(\Z_{\ell}):\Gamma] > I$.  So it suffices to bound $[\GL_2(\Z_{\ell}):G']$ while assuming that $G' \subset \Gamma$
for a \emph{fixed} $\Gamma$. \medskip

\noindent Step 2: Let $\Gamma \subset \GL_2(\Z_{\ell})$ be an open subgroup with
\begin{equation}
\label{KEYINDEXEQ}
[\GL_2(\Z_{\ell}):\Gamma] > I.
\end{equation}
Put
\[ S\Gamma \coloneqq \Gamma \cap \SL_2(\Z_{\ell}). \]
Then $\overline{\Gamma} \coloneqq \Gamma \cap \SL_2(\Z)$ is a congruence subgroup of $\SL_2(\Z)$.  Let
$P \overline{\Gamma}$ be the image of $\overline{\Gamma}$ in $\PSL_2(\Z)$, and put $D_{\Gamma} \coloneqq [\PSL_2(\Z):P \overline{\Gamma}]$.  Since $-1 \in G' \subset \Gamma$, we have $[\PSL_2(\Z):P \overline{\Gamma}] = [\SL_2(\Z):S \Gamma]$. Using (\ref{NEWSEQ}) we get
\[ D_{\Gamma} \geq \frac{[\GL_2(\Z_{\ell}):\Gamma]}{d}. \]

\noindent Step 3:  Let $\Q(\Gamma)$ be the finite subextension of $\Q^{\lcyc}/\Q$ corresponding
to the open subgroup $\det \Gamma \subset \Z_{\ell}^{\times}$.  Then there is a modular curve $X_{\Gamma}$ that is defined and geometrically integral over $\Q(\Gamma)$, and such that the base extension of $X_{\Gamma}$ to $\C$ is the compact Riemann surface $\overline{\Gamma} \backslash \overline{\mathcal{H}}$.  If for an elliptic curve $E_{/F}$ we have $G = \rho_{\ell^{\infty}}(\gg_{F}) \subset \langle \Gamma, -1 \rangle$, then $\Q(\Gamma) \subset F$, and there is an induced point on $X_{\Gamma}(F)$ and thus a closed point on $X_{\Gamma}$ of
degree dividing $d$.  Let $\mathfrak{d}_{\Gamma}$ be the gonality of the curve $X_{\Gamma}$ --
the least positive degree of a map $X_{\Gamma} \ra \PP^1$ defined over $\Q(\Gamma)$ -- and let $\overline{\mathfrak{d}}_{\Gamma}$ be the gonality of $(X_{\Gamma})_{/\C}$, so
\[ \overline{\mathfrak{d}}_{\Gamma} \leq \mathfrak{d}_{\Gamma}. \]
We claim that $X_{\Gamma}$ has only finitely many closed points
of degree dividing $d$.  Indeed, if not then by \cite[Thm. 5]{Clark09} we have
\[ \mathfrak{d}_{\Gamma} \leq 2d. \]
On the other hand, by a theorem of Abramovich \cite[Thm. 0.1]{Abramovich96}, we have
\[ \overline{\mathfrak{d}}_{\Gamma} \geq \frac{7}{800} D_{\Gamma}. \]
Putting these results together, we get the upper bound
\[ [\GL_2(\Z_{\ell}):\Gamma] \leq \frac{1600}{7} d^2, \]
and thus
\[ [\GL_2(\Z_{\ell}):\Gamma] \leq I, \]
contradicting (\ref{KEYINDEXEQ}).  Thus the set of $j$-invariants of non-CM elliptic curves over degree
$d$ number fields such that the index of the image of the $\ell$-adic Galois representation exceeds
$\frac{3200}{7} d^2$ is finite.  (Here we have multiplied by $2$ to get back from $G'$ to $G$.)  Let the exceptional $j$-invariants be $j_1,\ldots,j_M \in \overline{\Q}$.
For $1 \leq i \leq M$, choose an elliptic curve $E_{/\Q(j_i)}$ with $j$-invariant $j_i$, and let $I_i$ be the index
of the image of the $\ell$-adic Galois representation (by Serre's open image theorem, each $I_i$ is finite).  Now let $E_{/F}$ be a non-CM elliptic curve defined over a degree $d$ number field $F$
such that the index of the image of the $\ell$-adic Galois representation exceeds $\frac{3200}{7} d^2$.  Then
$j(E) = j(E_i)$ for some $i$.  There is a number field $K \supset F$ such that $E_{/K} \cong (E_i)_{/K}$
and $[K:\Q] \leq 2 d^2$, and thus the index of the image of the $\ell$-adic Galois representation of $E_{/F}$
is at most $2d^2 I_i$.  So for any non-CM elliptic curve over any degree $d$ number field for which the
image of the $\ell$-adic Galois representation is contained in $\Gamma$, the index of
the image of the $\ell$-adic Galois representation is at most
\[ \max \left(  \frac{3200}{7} d^2 , \max_{1 \leq i \leq M} 2d^2 I_i \right). \qedhere\]

\subsection{A consequence}
\begin{thm}
\label{ARAICOR}
Let $d _0 \in \Z^+$, and let $\mathfrak{s} \subset \mathcal{P}$ be finite. There is $C = C(d_0,\mathfrak{s}) \in \Z^+$ such that if $E_{/F}$ is a non-CM elliptic curve with $[\Q(j(E)):\Q] = d_0$, then
\[ \exp E(F)[\mathfrak{s}^{\infty}] \leq C [F:\Q(j(E))]^{\frac{1}{2}}. \]
\end{thm}
\begin{proof}
We consider non-CM elliptic curves $E$ defined over number fields $F$ such that $F_0 \coloneqq \Q(j(E))$
is a number field of degree $d_0$.  In case $F = F_0$, as we range over all $E_{/F_0}$, by Theorem \ref{STRONGARAITHM} the index of the image of the $\ell$-adic Galois representation
is uniformly bounded.  Thus there is $v_{\ell} \in \N$ such that for all such $E_{/F_0}$, we have
\[ \ord_{\ell} [\GL_2(\Z_{\ell}):\rho_{E,\ell^{\infty}}(\mathfrak{g}_{F_0})] \leq v_{\ell}. \]
With this notation, we will show that we may take
\[ C = \left( 2 \prod_{\ell \in \mathfrak{s}} \ell^{2+v_{\ell}} \right)^{\frac{1}{2}}. \]
Now let $E_{/F}$ be as above, and suppose $E(F)$ has a point of order
\[ N = \prod_{\ell \in \mathfrak{s}} \ell^{a_{\ell}}. \]
Let $\mathfrak{s}' = \{ \ell \in \mathfrak{s} \mid a_{\ell} > 0 \}$.  We also suppose, temporarily, that $E$ arises by base extension from
an elliptic curve defined over $F_0$.  Let $\ell \in \mathfrak{s}'$.  Since $E(F)$ has a point of
order $\ell^{a_{\ell}}$, this forces the image of the $\ell$-adic Galois representation to lie in a subgroup conjugate to
\[\Gamma_1(\ell^{a_{\ell}}) := \begin{bmatrix} 1 + \ell^{a_{\ell}} \Z_{\ell} & \Z_{\ell} \\
\ell^{a_{\ell}} \Z_{\ell} & \Z_{\ell}^{\times} \end{bmatrix}. \]
Since $[\GL_2(\Z_{\ell}):\Gamma_1(\ell^{a_{\ell}})] = \ell^{2a_{\ell}-2} (\ell^2-1)$, we get
\[  \ell^{2a_{\ell}-2} (\ell^2-1)  \mid
[\GL_2(\Z_{\ell}):\rho_{E,\ell^{\infty}}(\mathfrak{g}_{F})] \mid [F:F_0] [\GL_2(\Z_{\ell}):\rho_{E,\ell^{\infty}}(\mathfrak{g}_{F_0})]   \]
and thus
\[ \ell^{2 a_{\ell} -2 - v_{\ell}} \mid [F:F_0]. \]
(Here and below, we write $a\mid b$ for rational numbers $a, b$ whenever $b=aq$ for some $q \in \Z$.)
Compiling these divisibilities across all $\ell \in \mathfrak{s}'$, we get
\[ \frac{N^2}{\prod_{\ell \in \mathfrak{s}'} \ell^{2+v_{\ell}}} \mid [F:F_0] \]
and thus
\[ N \leq \left( \prod_{\ell \in \mathfrak{s}'} \ell^{2+v_{\ell}} \right)^{\frac{1}{2}} [F:F_0]^{\frac{1}{2}} \leq \left( \prod_{\ell \in \mathfrak{s}} \ell^{2+v_{\ell}} \right)^{\frac{1}{2}} [F:F_0]^{\frac{1}{2}}. \]
Now suppose that $E_{/F}$ does not necessarily arise by base extension from
an elliptic curve over $F_0$.  Nevertheless there is an elliptic curve $(E_0)_{/F_0}$ and
and a quadratic extension $F'/F$ such that $E_{/F'} \cong (E_0)_{/F'}$.  Since
$\exp E(F)[\mathfrak{s}^{\infty}] \mid \exp E(F')[\mathfrak{s}^{\infty}]$, applying the previously
addressed special case with $F'$ in place of $F$ gives
\[ N \leq \left( \prod_{\ell \in \mathfrak{s}} \ell^{2+v_{\ell}} \right)^{\frac{1}{2}} [F':F_0]^{\frac{1}{2}} = \left( 2 \prod_{\ell \in \mathfrak{s}} \ell^{2+v_{\ell}} \right)^{\frac{1}{2}}[F:F_0]^{\frac{1}{2}}. \qedhere\]
\end{proof}

\begin{remark}
Theorem \ref{ARAICOR} is sharp up to the value of the constant.  Indeed, for a non-CM elliptic curve $E$ defined
over a number field $F_0$ and any prime $\ell$, since there are $(\ell^2-1) \ell^{2n-2}$ points of order $\ell^n$
on $E(\overline{F})$, there is a field extension $F_n/F_0$ of degree at most
$(\ell^2-1)\ell^{2n-2} < \ell^{2n}$ such that $E(F_n)$ has a point of order $\ell^n$.
\end{remark}


\section{The proofs of Theorems \ref{THM1}, \ref{THM2} and \ref{THM3}}

\subsection{An easy lemma}

\begin{lemma}
\label{LEMMA1}
Let $F$ be a complete discretely valued field, with residue characteristic $p \geq 0$.  Let $\ell \neq p$ be a prime
number, and let $E_{/F}$ be an elliptic curve over $F$
with semistable reduction.  Then $e(F(E[\ell])/F) \mid \ell$.
\end{lemma}
\begin{proof}
Case 1: Suppose $E_{/F}$ has good reduction.  Then by
the N\'eron-Ogg-Shafarevich criterion, since $\ell \neq p$ the extension $F(E[\ell])/F$ is unramified: $e(F(E[\ell])/F) = 1$.
\medskip

\noindent Case 2: Suppose $E_{/F}$ has multiplicative reduction.  Then there is an unramified quadratic extension $F'/F$ such that $E_{/F'}$ admits an analytic uniformization, or in other words is a \textbf{Tate curve} in the sense of \cite[$\S 5.3$]{SilvermanII}: $E_{/F'} \cong \mathbb{G}_m/\langle q^{\mathbb{\Z}} \rangle$.  It follows that $F'(E[\ell]) = F'(\mu_{\ell},q^{\frac{1}{\ell}})$.  Put $F'' = F'(\mu_{\ell})$.
Since $F'/F$ is unramified and $\ell \neq p$, the extension $F''(q^{\frac{1}{\ell}})/F''$ is Galois and
\[ e(F(E[\ell])/F) = e(F''(q^{\frac{1}{\ell}})/F'') \mid
[F''(q^{\frac{1}{\ell}}):F''] . \]
By Kummer theory, we have
 \[ [F''(q^{\frac{1}{\ell}}):F''] = \begin{cases}
 1 & \text{if $q$ is an $\ell$th power in $F''$}, \\
 \ell & \text{otherwise}.
 \end{cases} \qedhere \]
 \end{proof}



\subsection{Bounding the exponent}
\noindent For the sake of a uniform presentation, we begin by fixing some notation. In the case of Theorems \ref{THM1} and \ref{THM2}, we let $F_0$ be as in the theorem statement, put $d_0=[F_0:\Q]$, and define $\ell_0$ as in Corollary \ref{KEYCORB}. In the case of Theorem \ref{THM3}, we let $F_0$ be any number field of the given degree $d_0$, and we define $\ell_0$ as in Corollary \ref{KEYCORA}.

Let $E$ be a non-CM elliptic curve over over a degree $d$ number field $F$ having $j(E) \in F_0$. (For CM curves, any of \cite{Silverberg92}, \cite{HS99}, \cite{CP15,CP17} yield stronger results.) It suffices to show that there is a constant $C$ with
\[ \exp E(F)[{\rm tors}] \le C d^{3/2+\epsilon}, \]
where $C$ is a function of $\ell_0$, $d_0$, and $\epsilon$. Note that since $\ell_0$ is determined by $F_0$ in the case of Theorems \ref{THM1} and \ref{THM2}, the constant $C$ depends on $F_0$ and $\epsilon$ in those theorems. Since $\ell_0$ is determined by $d_0$ in the case of Theorem \ref{THM3}, the constant $C$ depends only on $d_0$ and $\epsilon$ in that situation.

\medskip
\noindent Step 0: We first reduce to a special case. Suppose that there is a $C'>0$ with the property that for all elliptic curves $E_{/F}$ obtained by base extension from an elliptic curve $(E_0)_{/F_0}$, we have
\[ \exp E(F)[\tors] \leq C' [F:\Q]^{3/2+\epsilon}. \]
Now let $E_{/F}$ be an elliptic
curve with $j(E) \in F_0$.  Then there is an elliptic curve $(E_0)_{/F_0}$ with $j(E) = j(E_0)$ and a quadratic  extension $F'/F$ such that $E_{/F'} \cong (E_0)_{/F'}$.  Let $d = [F:\Q]$ and $d' = [F':\Q]$. Then
\[ \exp E(F)[\tors] \mid \exp E(F')[\tors] \leq C' \cdot d'^{3/2+\epsilon} \le (C'\cdot 2^{3/2+\epsilon}) \cdot d^{3/2+\epsilon}.
\]
So we may choose $C = C' \cdot 2^{3/2+\epsilon}$. If $C'$ depends only on $\ell_0, d_0$, and $\epsilon$, then so does $C$.

\medskip
\noindent {Step 1}: Now suppose that $E$ is obtained by base
change from an elliptic curve defined over $F_0$, which for notational simplicity we continue to denote by $E$.  Write \[E(F)[\tors] = \prod_{\ell \in \mathcal{P}} E(F)[\ell^{\infty}] \] and, for $\ell \in \mathcal{P}$,
\[E(F)[\ell^{\infty}] \cong \Z/\ell^{a_\ell} \Z \times \Z/\ell^{b_\ell} \Z \] for integers $a_{\ell}, b_{\ell}$ satisfying $0 \le a_{\ell} \le b_{\ell}$.  We will partition $\{\ell \in \mathcal{P} \mid b_{\ell} \geq 1\}$ into two classes $\Ss_1$ and $\Ss_2$, get bounds on $\exp E(F)[\Ss_1^{\infty}]$ and $\exp E(F)[\Ss_2^{\infty}]$, and multiply them to get a bound on $\exp E(F)[\tors]$. \medskip
\medskip

\noindent {Step 2}: Put $\Ss_1 = \{ \ell \in \mathcal{P} \mid \ell \leq \ell_0\}$.  By Theorem \ref{ARAICOR}, we have
\[ \exp E(F)[\Ss_1^{\infty}] \leq C_1(d_0,\ell_0) [F:F_0]^{\frac{1}{2}} \le C_1(d_0,\ell_0) d^{1/2}. \]
\medskip

\noindent {Step 3}: Let
\[ \Ss_2 = \{ \ell \in \mathcal{P} \mid \ell > \ell_0 \text{ and } b_{\ell} \geq 1\}. \]
List the elements of $\Ss_2$ in decreasing order:
\[ \ell_1 > \ell_2 > \dots > \ell_k. \]
Suppose that $\ell \in \Ss_2$ and $P \in E(F)$ is a point of order $\ell^{b_{\ell}}$. By Corollaries \ref{KEYCORA} and
\ref{KEYCORB}, for each prime $\mathfrak{l}$ of $F_0$ lying above $\ell$,  there is a positive integer $c\le 12d_0$ and a prime $\mathcal{L}$ of $F_0(P)$ lying above $\mathfrak{l}$ with $e(\mathcal{L}/\mathfrak{l})$ divisible by either $\varphi(\ell^{b_\ell})/\gcd(\varphi(\ell^{b_\ell}),c)$ or $\ell^{b_\ell}$. Thus, $e(\mathcal{L}/\ell)$ is divisible by either $\varphi(\ell^{b_\ell})/\gcd(\varphi(\ell^{b_\ell}),(12d_0)!)$ or by $\ell^{b_\ell}$.
\\ \\
For simplicity, from now on we write the exponent on $\ell_i$ as $b_{i}$ rather than $b_{\ell_i}$. Let $P_1, \dots, P_k$ be $F$-rational torsion points of orders $\ell_1^{b_1}, \dots, \ell_k^{b_k}$. Choose $D_1, \dots, D_k$ with
\[ D_i \in \{\varphi(\ell_i^{b_i})/\gcd(\varphi(\ell_i^{b_i}),(12d_0)!), \ell_i^{b_i}\} \]
and such that the field $F_0(P_i)$ has a prime $\mathcal{L}_i$ above $\ell_i$ with $e(\mathcal{L}_i/\ell_i)$ divisible by $D_i$.
\\ \\
We introduce the sequence of fields $K_{-1} = F_0$, $K_0 = F_0(E[12])$, $K_1 = K_{0}(P_1)$, $K_2 = K_1(P_2)$, \dots, $K_k = K_{k-1}(P_k)$. By a result of Raynaud, $E_{/K_0}$ has everywhere semistable reduction; see e.g. \cite[Thm. 3.5]{SZ95}. Since $K_k \subset F(E[12])$, we have
\begin{equation}\label{eq:productlessthand} \prod_{i=0}^{k} [K_i: K_{i-1}] = [K_k: F_0] \le [F(E[12]):\Q] = [F(E[12]):F] \cdot d \le 4608d. \end{equation}
(We used here that $\#\GL_2(\Z/12\Z) = 4608$.) Moreover,
\begin{equation}\label{eq:D1D0bound} [K_1: K_0][K_0:K_{-1}] = [K_1:F_0] \geq [F_0(P_1):F_0] = \frac{1}{d_0} [F_0(P_1):\Q] \ge \frac{D_1}{d_0}.\end{equation}
Now let us look at $[K_i: K_{i-1}]$ where $i \ge 2$. For notational simplicity, let $\ell=\ell_i$. Since $K_i \supset F_0(P_i)$, we know there is a prime $\mathcal{L}$ of $K_i$ above $\ell$ for which $e(\mathcal{L}/\ell)$ is divisible by $D_i$.  Define prime ideals $\mathcal{L}_{j}$ of $\Z_{K_j}$, for $j=-1,0, 1, \dots, i$, by
\[ \mathcal{L}_{j} = K_{j} \cap \mathcal{L}. \]
Thus, $\mathcal{L}_i = \mathcal{L}$, and
\begin{equation*} D_i \mid e(\mathcal{L}/\ell) = e(\mathcal{L}_{-1}/\ell) e(\mathcal{L}_0/\mathcal{L}_{-1}) \prod_{j=1}^{i} e(\mathcal{L}_{j}/\mathcal{L}_{j-1}). \end{equation*}
Since $e(\mathcal{L}_{-1}/\ell) \le d_0$ and $e(\mathcal{L}_0/\mathcal{L}_{-1}) \mid [F_0(E[12]):F_0] \mid 4608$, it follows that
\begin{equation}\label{eq:fundiv} D_i \mid d_0! \cdot 4608 \cdot \prod_{j=1}^{i} e(\mathcal{L}_{j}/\mathcal{L}_{j-1}). \end{equation}
Suppose that $1 \le j < i$. We have \[ K_{j} = K_{j-1}(P_j) \subset K_{j-1}(E[\ell_{j}^{b_{j}}]).\] Hence, $e(\mathcal{L}_{j}/\mathcal{L}_{j-1})$ divides the ramification index of $\mathcal{L}_{j-1}$ in $K_{j-1}(E[\ell_{j}^{b_{j}}])$. Note that $\mathcal{L}_{j-1}$ lies above the rational prime $\ell=\ell_i$, which is distinct from $\ell_j$ (since $j < i$). By Lemma \ref{LEMMA1}, the ramification index of $\mathcal{L}_{j-1}$ in $K_{j-1}(E[\ell_{j}])$ divides $\ell_j$. Moreover,
\[ [K_{j-1}(E[\ell_j^{b_j}]) : K_{j-1}(E[\ell_j])]  \]
is a power of $\ell_{j}$. (The image of the representation on the $\ell_j^{b_j}$-torsion lands in the kernel of the natural map $\mathrm{GL}_2(\Z/\ell_j^{b_j}\Z) \ra \mathrm{GL}_2(\Z/\ell_j\Z)$, an $\ell_j$-group.) Thus, the ramification index of $\mathcal{L}_{j-1}$ in $K_{j-1}(E[\ell_{j}^{b_j}])$ is a power of $\ell_{j}$.  So $e(\mathcal{L}_{j}/\mathcal{L}_{j-1})$ is also a power of $\ell_{j}$.
\\ \\
The definition of $D_i$ and the ordering of the primes $\ell_1, \dots, \ell_k$ imples that $D_i$ is coprime to $\ell_1, \dots, \ell_{i-1}$.  Since $e(\mathcal{L}_j/\mathcal{L}_{j-1})$ is a power of $\ell_j$ for $j < i$, \eqref{eq:fundiv} implies that
\[ D_i \mid d_0! \cdot 4608\cdot  e(\mathcal{L}/\mathcal{L}_{i-1}),\quad\text{whence}\quad e(\mathcal{L}/\mathcal{L}_{i-1}) \ge \frac{D_i}{4608\cdot d_0!}. \]
Therefore,
\[ [K_i: K_{i-1}] \ge e(\mathcal{L}_i/\mathcal{L}_{i-1}) \ge \frac{D_i}{4608\cdot d_0!} \qquad (\text{for }2 \le i \le k). \]
From the definition of $D_i$, it is easy to see that every
\[ D_i \ge \frac{\ell_i^{b_i}}{2\cdot (12d_0)!}. \]
Combining these estimates with \eqref{eq:productlessthand} and \eqref{eq:D1D0bound}, we find that
\begin{align}\label{eq:complicateddlower} 4608d \ge \prod_{i=0}^{k}[K_i:K_{i-1}] &= [K_0:K_{-1}][K_1:K_0] \prod_{2 \le i \le k}[K_{i}:K_{i-1}] \\&\ge \frac{D_1}{d_0} \cdot \prod_{2\le i \le k} \frac{D_i}{4608 \cdot d_0!}\ge \prod_{i=1}^{k}\frac{\ell_i^{b_i}}{9216 \cdot (12d_0)!^2}.\notag \end{align}
Put
\[ Z_0 = 9216 \cdot (12d_0)!^2. \]
Let $Z> Z_0$ be a constant (depending on $d_0$ and $\epsilon$) to be specified momentarily.  If $\ell_i^{b_i} > Z$, then the $i$th factor in the last displayed product on $i$ is at least $\frac{Z}{Z_0}$. Hence, there can be at most $\log(4608d)/\log{(Z/Z_0)}$ such values of $i$. There are at most $\pi(Z)$ values of $i$ with $\ell_i^{b_i} \le Z$, where $\pi(\cdot)$ is the usual prime-counting function. So
\[ k \le \pi(Z) + \frac{\log(4608d)}{\log(Z/Z_0)}.\] We may now deduce from \eqref{eq:complicateddlower} that
\[ \prod_{i=1}^{k}\ell_i^{b_i} \le Z_0^k \cdot 4608 d \le Z_0^{\pi(Z)} \cdot 4608^{1+\log(Z)/\log(Z/Z_0)} \cdot d^{1+\log(Z_0)/\log(Z/Z_0)}. \]
We fix $Z$ large enough, in terms of $d_0$ and $\epsilon$, to make $\log(Z_0)/\log(Z/Z_0) < \epsilon$; then
\[ \exp E(F)[\Ss_2^{\infty}] = \prod_{i=1}^{k} \ell_i^{b_i} \le C_2(d_0,\epsilon) d^{1+\epsilon}. \]\medskip
\noindent Step 4: Putting the contribution from $\Ss_1$ and $\Ss_2$ together,
\begin{align*} \exp E(F)[\tors] &\le C_1(d_0,\ell_0) d^{1/2} \cdot C_2(d_0,\epsilon) d^{1+\epsilon} \le C(d_0,\ell_0,\epsilon) d^{3/2+\epsilon},\end{align*}
as desired.

%

\subsection{From the exponent to the order}
\noindent Let $\mathcal{F}$ be a set, each element of which is an elliptic curve defined over a number field $E_{/F}$, and such that for some $B > 0$ and all $E_{/F} \in \mathcal{F}$ we have
\[ \exp E(F)[\tors] = O([F:\Q]^B). \]
(The implied constant is allowed to depend on $\mathcal{F}$, but \emph{not} on the choice of $E_{/F} \in \mathcal{F}$.) Let $b = \exp E(F)[\tors]$.  Then there is a positive integer $a$ dividing $b$ such that
\[ E(F)[\tors] \cong \Z/a\Z \oplus \Z/b\Z. \]
Since $E$ has full $a$-torsion over $F$, the field $F$ contains $\Q(\mu_a)$, so that $[F:\Q] \geq \varphi(a)$.  It follows (cf. \cite[Thms. 327, 328]{HW08}) that for all $\epsilon > 0$, we have
\[ a = O_{\epsilon}([F:\Q]^{1+\epsilon}). \]
 So
\[ \# E(F)[\tors] = ab = O_{\epsilon}([F:\Q]^{B+1+\epsilon}), \]
the implied constant depending on $\epsilon$ (and $\mathcal{F}$) but \emph{not} on the choice $E_{/F} \in \mathcal{F}$.
\\ \indent
Applying this with $B = \frac{3}{2}$ gives a bound $O_{\epsilon}([F:\Q]^{\frac{5}{2}+\epsilon})$ on the size of
the torsion subgroup, completing the proofs of Theorems \ref{THM1}, \ref{THM2} and \ref{THM3}.


\begin{thebibliography}{S-LMW}

\bibitem[Ab96]{Abramovich96} D. Abramovich,
\emph{A linear lower bound on the gonality of modular curves}. Internat. Math. Res.
Notices 1996, 1005--1011.


\bibitem[Ar08]{Arai08} K. Arai, \emph{On uniform lower bound of the Galois images associated to elliptic curves}. J. Th\'eor. Nombres Bordeaux 20 (2008), 23--43.


\bibitem[BC16]{BC16} A. Bourdon and P.L. Clark, \emph{Torsion points and Galois represenations on CM elliptic curves},
\url{http://arxiv.org/abs/1612.03229}.

\bibitem[Br10]{Breuer10} F. Breuer, \emph{Torsion bounds for elliptic curves and Drinfeld modules}.
J. Number Theory 130 (2010), 1241--1250.



\bibitem[Cl09]{Clark09} P.L. Clark, \emph{On the Hasse principle for Shimura curves}. Israel J. Math. 171 (2009),  349--365.

\bibitem[CMP17]{CMP17} P.L. Clark, M. Milosevic and P. Pollack, \emph{Typically bounding torsion}, preprint.

\bibitem[CP15]{CP15} P.L. Clark and P. Pollack, \emph{The truth about torsion in the CM case}.  C.R. Acad. Sci. Paris 353 (2016), 683--688.

\bibitem[CP17]{CP17} \bysame, \emph{The truth about torsion in the CM case, II}. To appear in the Quarterly Journal of
Mathematics.


\bibitem[CT12]{CT12} A. Cadoret and A. Tamagawa, \emph{A uniform open image theorem for $\ell$-adic representations, I.}
Duke Math. J. 161 (2012), 2605--2634.

\bibitem[CT13]{CT13} \bysame, \emph{A uniform open image theorem for $\ell$-adic representations, II.}
Duke Math. J. 162 (2013), 2301--2344.

\bibitem[De16]{Derickx16} M. Derickx, \emph{Torsion points on elliptic curves over number fields of small 
degree}. Leiden PhD thesis, 2016.


\bibitem[FJ]{FJ} M. Fried and M. Jarden, \emph{Field arithmetic}.
Third edition. Ergebnisse der Mathematik und ihrer Grenzgebiete, Springer-Verlag, Berlin, 2008.


\bibitem[Fr94]{Frey94} G. Frey,
\emph{Curves with infinitely many points of fixed degree}. Israel J. Math. 85 (1994), 79--83.


\bibitem[HS99]{HS99} M. Hindry and J. Silverman, \emph{Sur le nombre de points de torsion rationnels sur une courbe elliptique.} C. R. Acad. Sci. Paris 329 (1999), 97--100.

\bibitem[HW08]{HW08} G.H. Hardy and E.M. Wright, \emph{An introduction to the theory of numbers}. Sixth edition. Oxford University Press, Oxford, 2008.


\bibitem[La65]{Lazard65} M. Lazard, \emph{Groupes analytiques $p$-adiques}. Inst. Hautes Études Sci. Publ. Math. No. 26 (1965), 389--603.

\bibitem[LR15]{UNIFORM} \'A. Lozano-Robledo, \emph{Uniform boundedness in terms of ramification}. 
\url{http://alozano.clas.uconn.edu/wp-content/uploads/sites/490/2014/01/lozano-robledo_ramification_bounds_v33.pdf}.

\bibitem[LV14]{LV14} E. Larson and D. Vaintrob, \emph{Determinants of subquotients of Galois representations associated with abelian varieties.} With an appendix by Brian Conrad. J. Inst. Math. Jussieu 13 (2014), 517--559.

\bibitem[Ma78]{Mazur78} B. Mazur, \emph{Rational isogenies of prime degree (with an appendix by D. Goldfeld)}.
Invent. Math. 44 (1978),  129--162.

\bibitem[Me96]{Merel96} L. Merel, \emph{Bornes  pour  la  torsion  des  courbes  elliptiques  sur  les  corps  de  nombres}.
Invent. Math. 124 (1996), 437--449.

\bibitem[Mo95]{Momose95} F. Momose, \emph{Isogenies of prime degree over number fields.}
Compositio Math. 97 (1995), 329--348.

\bibitem[NS07]{NS07} N. Nikolov and D. Segal, \emph{On finitely generated profinite groups. I. Strong completeness and uniform bounds}. Ann. of Math. (2) 165 (2007), 171--238.

\bibitem[Pa99]{Parent99} P. Parent, \emph{Bornes effectives pour la torsion des courbes elliptiques sur les corps de nombres.}
J. Reine Angew. Math. 506 (1999), 85--116.

\bibitem[R-MO]{Rouse} J. Rouse (http://mathoverflow.net/users/48142/jeremy-rouse), \emph{What are the strongest conjectured uniform versions of Serre's Open Image Theorem?}, URL (version: 2015-05-02): \url{http://mathoverflow.net/q/203837}

\bibitem[S-LMW]{S-LMW} J.-P. Serre, \emph{Lectures on the Mordell-Weil theorem.} Third edition. Aspects of Mathematics, Friedr. Vieweg \& Sohn, Braunschweig, 1997.


\bibitem[S-MG]{S-MG} \bysame, \emph{Abelian $\ell$-adic representations and elliptic curves}. Benjamin, New York, 1968.


 \bibitem[Se72]{Serre72} \bysame, \emph{Propri\'et\'es galoisiennes des points d'ordre fini des courbes elliptiques}. Invent. Math.  15  (1972), 259--331.


\bibitem[Si88]{Silverberg88} A. Silverberg, \emph{Torsion points on abelian varieties of CM-type}.  Compositio Math.  68  (1988), 241--249.

\bibitem[Si92]{Silverberg92} \bysame,  \emph{Points
of finite order on abelian varieties}.  In \emph{$p$-adic methods
in number theory and algebraic geometry}, 175--193,
Contemp. Math. 133, Amer. Math. Soc., Providence, RI, 1992.

\bibitem[Si94]{SilvermanII} J. Silverman, \emph{Advanced Topics in the Arithmetic of Elliptic Curves}. Graduate Texts in Mathematics 151, Springer-Verlag, 1994.

\bibitem[SZ95]{SZ95} A. Silverberg and Yu. G. Zarhin, \emph{Semistable reduction and torsion subgroups of abelian varieties}. Ann. Inst. Fourier (Grenoble) 45 (1995), 403--420.


\bibitem[Zy15]{Zywina15} D. Zywina, \emph{On the possible images of the mod $\ell$ representations associated to elliptic curves over $\Q$}, \url{http://arxiv.org/abs/1508.07660}.




\end{thebibliography}
\end{document}